\documentclass[12pt]{amsart}
\usepackage{amssymb,latexsym}
\usepackage{enumerate}
\usepackage[all]{xy}
\usepackage{graphicx}
\makeatletter
\@namedef{subjclassname@2010}{%
  \textup{2010} Mathematics Subject Classification}
\makeatother
\newtheorem{theorem}{Theorem}[section]
\newtheorem{corollary}[theorem]{Corollary}
\newtheorem{lemma}[theorem]{Lemma}
\newtheorem{proposition}[theorem]{Proposition}
\theoremstyle{definition}
\newtheorem{definition}[theorem]{Definition}

\newtheorem{example}[theorem]{Example}
\newtheorem*{xrem}{Remark}

\newcommand{\st}{\stackrel}
\newcommand{\sub}{\subseteq}

\newcommand{\ov}{\overline}

\newcommand{\lo}{\longrightarrow}

\newcommand{\hin}{\widetilde{\in}}
\newcommand{\wt}{\widetilde}
\newcommand{\vf}{\varphi}
\newcommand{\fr}{\frac}
\newcommand{\al}{\alpha}

\numberwithin{equation}{section}
\begin{document}
\baselineskip=17pt
\title[On H-groups and their applications]
{On H-groups and their applications to topological homotopy groups}
\author[A. Pakdaman]{Ali Pakdaman}
\address{Department of Pure Mathematics\\ Ferdowsi University of Mashhad\\
P.O.Box 1159-91775, Mashhad, Iran.}
\email{Alipaky@yahoo.com}
\author[H. Torabi]{Hamid Torabi}
\address{Department of Pure Mathematics\\ Ferdowsi University of Mashhad\\
P.O.Box 1159-91775, Mashhad, Iran.}
\email{h$ _{-}$torabi86@yahoo.com}

\author[B. Mashayekhy]{Behrooz Mashayekhy}
\address{Department of Pure Mathematics\\ Center of Excellence in Analysis on Algebraic Structures\\ Ferdowsi University of Mashhad\\
P.O.Box 1159-91775, Mashhad, Iran.}
\email{bmashf@um.ac.ir}
\date{}
\begin{abstract}
This paper develops a basic theory of H-groups. We
introduce a special quotient of H-groups and
extend some algebraic constructions of topological groups to the category
of H-groups and H-maps. We use these constructions to prove some advantages in topological homotopy groups. Also, we present a family of spaces that their topological fundamental groups are indiscrete topological group and find out a family of spaces whose topological fundamental group is a topological group.
\end{abstract}



\maketitle

\section{Introduction and Motivation}
An H-group is a homotopy associative H-space with a given homotopy inverse. There are two main classes of motivating examples of H-groups. The first is the class of topological groups and the second is the class of loop spaces. Topological groups have been studied from a variety of viewpoint. Specially there is an enriched developed basic theory for topological groups similar to abstract group theory.
However, it seems that there is no such a basic theory for H-groups. One can find only the concept of sub-H-group of an H-group in \cite{KP} and some elementary properties in \cite{D}. One of the main objects in this paper is to develop a basic theory for H-groups similar to abstract group theory.

On the other hand, loop spaces have the main role in homotopy groups especially in topological homotopy groups \cite{G1}. The topological n-th homotopy group of a pointed space $(X,x)$ is the familiar n-th homotopy group $\pi_n(X,x)$ by endowing a topology on it as a quotient of the n-loop space $\Omega^n(X,x)$ equipped with the compact-open topology, denoted by $\pi_n^{top}(X,x)$ \cite{G1}. The other main object of this paper is to study H-groups in order to get some applications to topological homotopy groups.

After giving main definitions and notations in section 2, we introduce in section 3 cosets of a sub-H-group, a normal sub-H-group and a quotient of an H-group in order to provide preliminaries to begin a basic theory for H-groups similar to elementary group theory. We develop the theory in section 4 by introducing the kernel of an H-homomorphisms in order to give H-isomorphism theorems, and the concept of the center and inner H-automorphism of an H-group and their relationship. In section 5, we give a topology to a quotient of H-groups which makes them quasitopological groups. We also study path component space of H-groups and give a necessary and sufficient condition for significance of 0-semilocally simply connectedness. Finally in section 6, we provide some applications of the basic theory of H-groups to topological homotopy groups. More precisely, among reproving some of the known results, by using advantages of section 5, we give some new results for discreteness and indiscreteness of $\pi_n^{top}(X,x)$, for $n\geq1$.
Also, we find out a family of spaces by using n-Hawaiian like spaces, introduced in \cite{G2}, such that their topological fundamental groups are indiscrete topological groups.

\section{Notations and preliminaries}
 An H-group consists of a pointed topological space $(P,p_0)$ together with
continuous pointed maps $\mu: P\times P\longrightarrow
P$, $\eta:P\longrightarrow P $ and the constant map
$c:P\longrightarrow P$, for which
$\mu(1_P,c)\simeq \mu(c,1_P)\simeq 1_P,\ \mu(\eta,1_P)\simeq
\mu(1_P,\eta)\simeq c$ and $\mu(\mu,1_P)\simeq\mu(1_P,\mu)$
(all of maps and homotopies are pointed ). The maps $\mu$, $\eta$ and $c$ are called multiplication, homotopy inverse and homotopy identity, respectively. For example every topological group and every loop space is an H-group. $P$ is called an Abelian H-group if $\mu\circ T\simeq\mu$, where $T:P\lo P$ by $T(x,y)=(y,x)$.

\begin{definition}
(\cite[XIX, 3]{D}). A continuous map $\varphi:P\longrightarrow P'$ for $(P,\mu,\eta,c)$ and $(P',\mu',\eta',c')$ as H-groups, is called an {\it H-homomorphism} whenever $\varphi\mu\simeq\mu'(\varphi\times\varphi)$ and $\varphi\eta\simeq\eta'\varphi$. Also, $\vf$ is called an {\it H-isomorphism} if there exists an H-homomorphism $\psi:P'\lo P$ such that $\vf\circ\psi\simeq1_{P'}$ and $\psi\circ\vf\simeq1_{P}$; in this event, the H-structures are called H-isomorphic.
\end{definition}
\begin{example}
(\cite[XIX, 3]{D}). Let $x,y\in X$, and let $\al$ be any path from $x$ to $y$. The map $\al^+:\Omega(X,x)\lo\Omega(X,y)$ by setting $\al^+(\beta)=\al^{-1}*(\beta*\al)$ is an H-isomorphism by $(\al^{-1})^+:\Omega(X,y)\lo\Omega(X,x)$ as inverse. Also, for each continuous map $f:(X,x)\lo (Y,y)$, $\Omega f:\Omega(X,x)\lo \Omega(Y,y)$ by $(\Omega f)(\al)=f\circ\al$ is an H-homomorphism and if $f$ is a homotopy equivalence, $\Omega f$ is an H-isomorphism.
\end{example}
\begin{proposition}(\cite[Theorem 7.2]{D}). If $P$ is an H-group, then $\pi_0(P)$ is a group with the equivalence class of $p_0$ as the identity. Also, for any H-homomorphism $\vf:P\lo P'$, $\pi_0(\vf):\pi_0(P)\lo\pi_0(P')$ is a group homomorphism.
\end{proposition}
\begin{definition}
(\cite[Definition 3.1]{KP}). A pointed subspace $P'$ of an H-group $(P,\mu,\eta,c)$ is called a sub-H-group of P if $P'$ is itself an
H-group such that the inclusion map $i:P'\lo P$ is an H-homomorphism.
\end{definition}
\begin{example}
Given a pointed space $(Y,y_0)$ with $(Y',y_0)$ as a pointed
subspace. Then the loop space $\Omega (Y',y_0)$ is a
sub-H-group of the loop space $\Omega (Y,y_0)$.
\end{example}
\begin{theorem}(\cite[Proposition 3.8]{KP}). If $P'$ is a sub-H-group of an H-group $(P,\mu,\eta,c)$, then\\
(i) There exists a continuous multiplication $\mu':P'\times
P'\longrightarrow P' $ such that $i\mu'\simeq \mu(i\times i)$; \\
(ii) For the constant map $c':P'\longrightarrow P'$ we have $ic'=ci$; \\
(iii) There exists a continuous map $\eta' :P'\longrightarrow P'$ such
that $i\eta'\simeq\eta i$.
\end{theorem}
Let $hTop^{*}$ be the category of pointed topological spaces with
homotopy class of pointed maps as morphism. Therefore a map
$f:(X,x_0)\longrightarrow(Y,y_0)$ is monic if and only if the only
pairs $g_1,g_2:(Z,z_0)\longrightarrow(X,x_0) \ $such that
$fg_1\simeq fg_2 $ are the homotopic ones: $g_1\simeq g_2 $. Also a map $f':(X,x_0)\longrightarrow(Y,y_0)$ is epic if and only if
the only pairs $h_1,h_2:(Y,y_0)\longrightarrow(Z',z'_0) \ $such that
$h_1f'\simeq h_2f' $ are the homotopic ones: $h_1\simeq h_2 $.
\begin{theorem}(\cite[Proposition 3.9]{KP}). Let $P'$ be a pointed subspace of an H-group $(P,\mu,\eta,c)$. Suppose that the
statements (i), (ii) and (iii) given in Theorem 2.6 are satisfied and the inclusion map $i: P'\lo P$
is a monic. Then $P'$ is a sub-H-group of $P$.

\end{theorem}
\section{ Quotient H-groups}
 In this section we assume that $(P,\mu,\eta,c) $ is an H-group and $(P',\mu',\eta',c')$ is a sub-H-group of $P$.
 Also from now on we use $g^{-1}$ and $g_1.g_2$ instead of $\eta(g)$ and $\mu(g_1,g_2)$, respectively, for abbreviation of notation.
\begin{definition}
For each $g\in P$ we say that $g$ homotopically belongs to $P'$, denoted by $g\widetilde{\in}P' $ if and only if there exists a path
$\alpha:I\longrightarrow P $ such that $\alpha(0)=g$ and
$\alpha(1)\in P'$.
 \end{definition}
 \begin{definition}
Let $P'$ be a sub-H-group of $P$ and $g\in P$. Then we define the right coset of
$P'$ with representative $g$ as follows: $$P'g=\{ g'\in P\ |\ g'.g^{-1}\hin P' \}.$$
Similarly, the left coset of $P'$ with representative $g$ is defined as follows: $$gP'=\{ g'\in
P\ |\ g^{-1}.g'\widetilde{\in}P' \}.$$
 \end{definition}
Notation: If $\al$ is a path in an H-group $P$ and $g\in P$, then by $g\al$ we mean the path $g\al:I\lo P$ given by $g\al(t)=g.\al(t)$.
 \begin{lemma}
 For every sub-H-group $P'$ of P the following statements hold:\\
(i) For each $g\ \in P$, $g\in gP'$.\\
(ii) If $g_1,g_2 \ \widetilde{\in}\ P'$, then $g_1^{-1}\
\widetilde{\in}P'$ and  $g_1.g_2\ \widetilde{\in}P'$.\\
(iii) $g_2^{-1}.g_1\hin P'$  if and only if $g_1P'=g_2P'$.\\
(iv) If $g_1$ and $g_2$ are in the same path component of $P$, then
$g_1P'=g_2P'$.
\end{lemma}
\begin{proof}
(i) Let $p_0$ be the common base point of $P$ and $P'$, then by $\mu(\eta,1)\simeq\ c$ there exists a path from
$g^{-1}.g$ into $p_0$. Hence $g\in gP'$.\\
(ii) Since $g_1\widetilde{\in}\ P'$, there exists a path
$\alpha$ from $g_1$ into $P'$, which implies $\eta \circ \alpha$ is a path
from $\eta(g_1)$ to $\eta (\al(1))\in \eta(P')$. Put $g'=\al(1)\in P'$, then $\eta
\alpha(1)=\eta (g')=\eta i(g')$. Since $i\eta'\simeq \eta i$, there exists a path $\beta$ from $\eta i(g')$ to
$i \eta'(g')$. Thus $(\eta\alpha)\ast\beta$ is a path from
$\eta(g_1)$ to $i\eta'(g')$ and since $\eta'(g')\in P'$, we have
$g_1^{-1}\ \widetilde{\in}\ P'$. By a similar method it can be shown that $g_1.g_2\
\widetilde{\in}P'$.\\
(iii) If $g_1P'=g_2P'$, then by (i) $g_1\in g_2P'$. Thus
 $g_2^{-1}.g_1\widetilde{\in}\ P'$. Conversely, let $g'\in g_1P'$, then $g_1^{-1}.g'\widetilde{\in}\ P'$. By (ii) we
 have $(g_2^{-1}.g_1).(g_1^{-1}.g')\widetilde{\in}\ P'$. Associativity of $\mu$ gives $g_2^{-1}.g'\widetilde{\in}\
 P'$ which implies $g'\in g_2P'$ and $g_1P'\subseteq g_2P'$. Similarly $g_2P'\subseteq
 g_1P'$.\\
(iv) If $\alpha $ is a path from $g_1$ to $g_2$, then
$g_1^{-1}\alpha$ is a path from $g_1^{-1}.g_1$ to
$g_1^{-1}.g_2$. In the proof of (i) we showed that
$g_1^{-1}.g_1$ is connected to $p_0$ by a path. Hence $g_1^{-1}.g_2$ is connected to $p_0$ by a path which implies
$g_1^{-1}.g_2\widetilde{\in}P'$.
\end{proof}
\begin{example}
Consider loops $\al_1,\al_2,\al_3,\al_4$ in $R^2$ as in Figure (1). If $X=\bigcup_{i=1}^4 Im(\al_i)$ and $Y=Im(\al_1)\bigcup Im(\al_3)$, then $\Omega Y$ is a sub-H-group of $\Omega X$ and $\al_1,\al_3\in \Omega Y$. Hence $\al_1\Omega Y=\al_3\Omega Y$, but $\al_1$ is not homotopic to $\al_3$. This shows that the converse of (iv) in the above theorem does not hold.

\begin{figure}
 \includegraphics[scale=0.5]{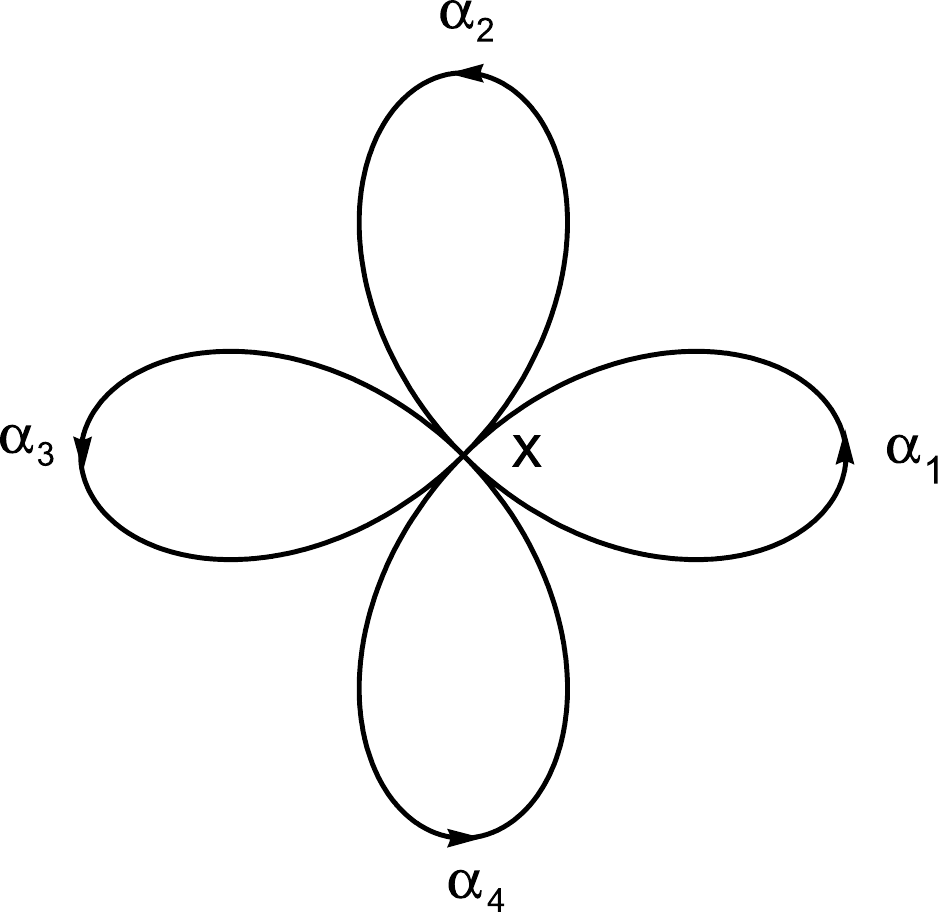}
  \caption{}\label{1}
\end{figure}
\end{example}
\begin{proposition}For each sub-H-group $P'$ of $P$, the relation $\stackrel{P'}{\sim}$ on $P$ defined by $$g_1\stackrel{P'}{\sim}g_2 \Leftrightarrow g_1.g_2^{-1}\hin P'$$
is an equivalence relation in which $g_1P'$ is the equivalence class of $g_1$.
\end{proposition}
\begin{proof}
Trivially $g\stackrel{P'}{\sim}g$ for each $g\in P$. If $g_1\stackrel{P'}{\sim}g_2$, then by Lemma 3.3 (iii) $g_2\stackrel{P'}{\sim}g_1$. Transitivity comes from definition and Lemma 3.3 (ii). The last assertion follows from Lemma 3.3 (i) and (iii).
\end{proof}
Note that by the above results the set of all left (right) cosets of $P'$ is a partition for $P$.
\begin{definition}
For a topological space $X$, we call a subset $A$ of $X$ saturated if for each $x\in A$ the path component of $X$ which contains $x$ is a subset of $A$. If $A\sub X$ is not saturated, then it's saturation in $X$ is defined as $\wt{A}=\{x\in X\ |\ x\hin A\}$.
\end{definition}
\begin{lemma} If $P'$ is a saturated sub-H-group of $P$, then for any $g\in P$ we have $$|\pi_0(P')|=|\pi_0(gP')|.$$
\end{lemma}
\begin{proof}We claim that if $g_1,g_2\in P'$ lie in different path
components of $P'$, then so do $g.g_1$ and $g.g_2$, since if there
exists a path $\alpha:I\longrightarrow gP'$ from $g.g_1$ to $g.g_2$,
then $g^{-1}\alpha$ is a path in $P'$ from $g^{-1}.(g.g_1)$ to
$g^{-1}.(g.g_2)$. By associativity of multiplication we have a path
from $g_1$ to $g_2$. Thus $|\pi_0(P')|\leq |\pi_0(gP')|$.\\
Similarly, if $g_1,g_2\in gP'$ do not lie in the same path component,
then $g^{-1}.g_1,g^{-1}.g_2\widetilde{\in}P'$ do not lie in the same
path component of $P'$. Hence $|\pi_0(gP')|\leq |\pi_0(P')|.$
\end{proof}
\begin{xrem}By homotopy associativity of $\mu$, we have $(g_1.g_2).g_3\
\stackrel{P'}{\sim}\ g_1.(g_2.g_3)$ for each sub-H-group $P'$ of $P$ and any $g_1,g_2,g_3\in P$.
\end{xrem}
 For a sub-H-group $P'$ of H-group $P$, there are as many right cosets as left cosets, since the
map $gP'\longrightarrow P'g^{-1}$ is a one-to-one correspondence.
If $[P : P']$, the index of $P'$ in $P$, denotes the cardinal of the set of all
left (or right) cosets of $P'$ in $P$, then we have the following basic result which is analogues to Lagrange theorem in group theory.
\begin{theorem} If $P'$ is a saturated sub-H-group of $P$, then $$|\pi_0(P)|=|\pi_0(P')|.[P;P'],$$
when $|\pi_0(P)|$ is finite.
\end{theorem}
\begin{proof}There are $[P : P']$ cosets, each of which with $|\pi_0(P')|$ members by Lemma 3.7. It is easy to check that
if left cosets $g_1P'$ and $g_2P'$ are different, then for each $g'\in
P'$, $g_1.g'$ and $g_2.g'$ are not in the same path component.
\end{proof}
\begin{proposition}If $P'$ is a saturated sub-H-group of $P$, then $\pi_0(P')$ is a subgroup of $\pi_0(P)$.
\end{proposition}
\begin{proof}
Let $[g_1],[g_2]\in \pi_0(P')$, then $g_1,g_2\hin P'$ gives $g_1.g_2^{-1}\hin P'$ by Lemma 3.3. Hence $[g_1][g_2]^{-1}=[g_1.g_2^{-1}]\in \pi_0(P')$.
\end{proof}
\begin{example}
Note that in the above proposition the hypothesis ``saturated'' is essential, for if $X=R^2$ and $Y$ is as in Example 3.4, then $\pi_0(\Omega(X))=1$ and $\pi_0(\Omega(Y))=Z*Z$, where $1$ is trivial group and $Z*Z$ is the free product of two copies of $Z$.
\end{example}
\begin{theorem}
A saturated subset $A$ of $P$ which is closed under inherited multiplication and inversion is a sub-H-group of $P$.
\end{theorem}
\begin{proof}
Let $\mu_A=\mu|_{A\times A}$ and $\eta_A=\eta|_A$ be as multiplication and inversion of $A$. By Theorem 2.7, it suffices to show that $i:A\lo P$ is monic. Let $h_1,h_2:Z\lo A$ such that $i\circ h_1\simeq i\circ h_2$ by a homotopy $H:Z\times I\lo P$. Since $H(z,0),H(z,1)\in A$ and path components of $A$ and $P$ coincide, $H(z,t)\in A$, for all $z\in Z, t\in I$. Hence the result holds.
\end{proof}
\begin{theorem}
For every subgroup $K$ of $\pi_0(P)$, there exists a sub-H-group $P_K$ of $P$ such that $\pi_0(P_K)=K$.
\end{theorem}
\begin{proof}
Define $P_K=\{g\in P|\ [g]\in K\}$, we show that $P_K$ is a sub-H-group of $P$.
Let $x,y\in P_K$, then $[x],[y]\in K$. Since $K$ is a subgroup of $\pi_0(P)$, $[x][y]=[x.y]\in K$ and $[x]^{-1}=[x^{-1}]\in K$ which implies $x.y,x^{-1}\in P_K$. Therefore $P_K$ is closed under inherited multiplication and inversion. Hence by Theorem 3.11 $P_K$ is a sub-H-group.
\end{proof}
The following corollary is a consequence of definitions.
\begin{corollary}
Let $P'$ be a saturated sub-H-group of $P$ and $K=\pi_0(P')$, then $P'=P_K$.
\end{corollary}
Suppose that $A$ and $B$ are two subsets of $P$, by definition we have $\wt{AB}=\{p\in P\ |\ p\hin AB\}$, where $AB=\{a.b\ |\ a\in A,\ b\in B\}$.
Note that if $P'$ and $P''$ are sub-H-groups of $P$, then $\widetilde{P'P''}$ do not need to be an H-group, since $p'_1.p''_1.p'_2.p''_2$ is not necessarily connected to $p'_1.p'_2.p''_1.p''_2$ by a path, where $p_1',p_2'\in P'$ and $p_1'',p_2''\in P''$. For example, consider $Z=Im(\al_4)$ in Example 3.4, then $\Omega Z$ and $\Omega Y$ are sub-H-groups of $\Omega X$ and $\al_1*\al_4*\al_3*\al_4$ is not homotopic to $\al_1*\al_3*\al_4*\al_4$. But if $P$ is an Abelian H-group, then $P'P''$ will be an H-group, and we have the following
useful generalization of this observation.
\begin{proposition}
If $P'$ and $P''$ are saturated sub-H-groups of $P$, then $\wt{ P'P''}$ is a sub-H-group of $P$ if and only if $\widetilde{P'P''}=\widetilde{P''P'}.$
\end{proposition}
\begin{proof}
 Put $H=\pi_0(P')$ and $K=\pi_0(P'')$, it is easy to check that $\pi_0(\wt{P'P''})=HK$. Since $\wt{P'P''}$ is saturated, by Corollary 3.13 and the similar algebraic fact, the result holds.
\end{proof}
\begin{lemma}
If $P'$ and $P''$ are sub-H-groups of $P$, then the following statements hold.\\
(i) $\wt{P'P'}=\wt{P'}$.\\
(ii) $gP'=g\wt{P'}=\wt{gP'}$, for each $g\in P$.\\
(iii) $P'P''\sub \wt{P'}\wt{P''}\sub \wt{P'P''}$.\\
(iv) $\wt{(g_1P')(g_2P')} = g_1(P'g_2)P'$, for each $g_1,g_2\in P$.\\
(v) $g'P'=\wt{P'}$, for each $g'\hin P'$.
\end{lemma}
\begin{proof}
For (iv) let $x\in\wt{(g_1P')(g_2P')}$, then $x$ is connected to $x_1.x_2$, where $g_1^{-1}.x_1,g_2^{-1}.x_2\hin P'$. Let $y=g_1^{-1}.x_1.g_2$. Thus $y.g_2^{-1}\hin P'$ and $y^{-1}.g_1^{-1}x$ is connected to $g_2^{-1}.x_2\hin P'$ which implies $x\in g_1(P'g_2)P'$. Conversely, for $x\in g_1(P'g_2)P'$, there exists $y\in P$ such that $y.g_2^{-1}\hin P'$ and $y^{-1}.g_1^{-1}.x\hin P'$. If $p_1'\in P'$ is an element that is connected to $y^{-1}.g_1^{-1}.x$ by a path, then put $x_1=g_1.y.g_2^{-1}$ and $x_2=g_2.p_1'$. The rest of statements can be easily proved.
\end{proof}
\begin{xrem}
Although cosets of every subset of $P$ are saturated, but if $A$ and $B$ are saturated subset of $P$, then $AB$ is not necessarily equal to $\wt{AB}$. For example, let $A=\Omega (Im(\al_4))$ and $B=\Omega Y$, as in Example 3.4, then $\al_4\in\wt{AB}$, but $\al_4$ can not be written as multiplication of two loops. Also let $P'=\Omega Y$, $g_1=\al_2$ and $g_2=\al_4$, then $\al_2*\al_1*\al_4*\al_3\in \wt{(g_1P')(g_2P')}$, but $\al_2*\al_1*\al_4*\al_3\notin (g_1.g_2)P'$.
\end{xrem}
The notation $\wt{P'P''}$ suggests a binary operation on cosets. If $P'$
is a sub-H-group of $P$, we can multiply $g_1P'$ and $g_2P'$, and it is natural that hope to get $g_1g_2P'$.
But this does not always happen as we have shown in above remark. Here is one possible criterion.
\begin{lemma}
If $P'$ be a sub-H-group of $P$, then the following two conditions are equivalent.\\
(i) $\wt{(g_1P')(g_2P')}=(g_1.g_2)P'$, for all $g_1,g_2\in P$; \\
(ii) $gP'=P'g$ (or equivalently $g^{-1}P'g=\wt{P'}$), for all $g\in P$.
\end{lemma}
\begin{proof}
Let (ii) holds, then by Lemma 3.15 we have:\\
 $\wt{(g_1P')(g_2P')} =g_1({P'}g_2)P'=g_1g_2{P'}{P'}=g_1g_2\wt{P'P'}= g_1g_2\wt{P'}=g_1g_2P'.$\\
  Conversely, if (i) holds, then  $g^{-1}P'g\subseteq \wt{(g^{-1}P'g)P'}$, since $p_0\in P'$, and $\wt{(g^{-1}P')(gP')} =
g^{-1}gP'(= \wt{P'})$ by hypothesis. Thus $gP'g^{-1}\subseteq \wt{P'}$ which implies that $gP'\subseteq \wt{P'}g=P'g$. Since this
holds for all $g\in P$, we have $P'g\sub gP'$, and the result follows.
\end{proof}
Note that if $gP'g^{-1}\subseteq \wt{P'}$, for all $g\in P$, then in fact $gP'g^{-1}=\wt{P'}$,
for all $g\in P$.
\begin{definition}
Let $P'$ be a sub-H-group of $P$. Then we call $P'$ a normal sub-H-group of $P$, denoted by $P'\unlhd P$, if and only if $g.g'.g^{-1}\hin P'$ for each $g\in P$ and $g'\in P'$. We define the quotient of $P$ by $P'$, denoted by $P/P'$ as follows: $$P/P'=\{gP'\ |\ g\in P\}.$$
\end{definition}
\begin{theorem}
If $P'$ is a normal sub-H-group of $P$, then
$P/P'$ is a group in which the coset $p_0P'(=\wt{P'})$ is the identity element.
\end{theorem}
\begin{proof}Define $(g_1P').(g_2P')=(g_1.g_2)P'$ and $(g_1P')^{-1}=g_1^{-1}P'$.
If $g_1P'=g_2P'$ and $h_1P'=h_2P'$, then $g_1^{-1}.g_2\widetilde{\in}
P'$ and $h_1^{-1}.h_2\widetilde{\in} P'$. Normality of $P'$
guaranties that $h_1^{-1}.(g_1^{-1}.g_2).h_1\widetilde{\in} P'$. By
associativity of $\mu$ and Lemma 3.3 (iii) we have
$$[h_1^{-1}.(g_1^{-1}.g_2).h_1].h_1^{-1}.h_2\widetilde{\in}
 P'\Rightarrow (h_1^{-1}.g_1^{-1}).(g_2.h_2)\widetilde{\in}
 P'\Rightarrow (g_2.h_2)P'=(g_1.h_1)P'.$$
 Therefore the product of two cosets does not depend on representatives.\\
{\it Associativity} : We must show that
$(g_1P'.g_2P').g_3P'=g_1P'.(g_2P'.g_3P')$ or equivalently
$(g_1.g_2).g_3P'=g_1.(g_2.g_3)P'$. By Lemma 3.3 (iii) it suffices to
show that $(g_1.g_2).g_3 \stackrel{P'}{\sim} g_1.(g_2.g_3)$, and this
holds by associativity of $\mu$. \\
{\it Inversion}: By definition of inverse we have
$gP'.g^{-1}P'=(g.g^{-1})P'$. Since $g.g^{-1}$ is
connected to $p_0$ by a path, $(g.g^{-1})P'=p_0P'\wt{P'}$\\
{\it Identity} : It is easy to see that $g.p_0\stackrel{P'}{\sim} g$, thus
$gP'.p_0P'=(g.p_0)P'=gP'$.\\
\end{proof}
\begin{corollary}
If $P'$ is a normal sub-H-group of $P$, then $P/P'\cong P/\wt{P'}$.
\end{corollary}
\begin{proof}
Using Lemma 3.15 (ii) and the fact that $\wt{P'}$ is identity element of $P/P'$, the result holds.
\end{proof}
\begin{lemma}Let $P'$ be a sub-H-group of $P$ and $P''$ be a
sub-H-group of $P'$. If $g_1\stackrel{P''}{\sim} g_2$, then
$g_1\stackrel{P'}{\sim} g_2$ and $g.g_1\stackrel{P'}{\sim}\ g.g_2$ for all $g_1,g_2,g\in P$.
\end{lemma}
\begin{proof}
Using definitions and associativity of $\mu$, the result holds.
\end{proof}
\begin{theorem}
If $P'$ is a sub-H-group of $P$ and $P''$ is a
sub-H-group of $P'$, then the following statements hold.\\
(i) If $P'$ is a normal saturated sub-H-group, then $\pi_0(P')$ is a normal subgroup of $\pi_0(P)$.\\
(ii) $P''$ is a sub-H-group of $P$.\\
(iii) If $P''$ is normal in $P$, then it is normal in $P'$.\\
(iv) If $P''$ is normal in $P$, then $P'/P''$ is a subgroup
of $P/P''$. Also, ${P'}/{P''}$ is a normal subgroup of
$P/P''$ if and only if $P'$ is a normal sub-H-group of $P$.
\end{theorem}
\begin{proof}
Using definitions and Lemma 3.20, the results hold.
\end{proof}
\begin{lemma}
The path component of $P$ that contains $p_0$, named principle component of $P$ which is denoted by $P_0$, is a normal sub-H-group of $P$ and $\pi_0(P)\simeq P/P_0$.
\end{lemma}
\begin{proof}
The first claim can be founded in \cite{D}. For the second, define $\theta:\pi_0(P)\lo P/P_0$ by $\theta([g])=gP_0$ which is easily an isomorphism.
\end{proof}
\section{H-morphisms}
\begin{definition}Let $\varphi:P\longrightarrow Q$ be an H-homomorphism. We define the
$kernel$ of $\varphi$ as $$ker\varphi=\{g\in P|\
\varphi(g)\widetilde{\in}\ Q_0\},$$
where $Q_0$ is the principle component of $Q$.
\end{definition}
\begin{proposition}
Let $\varphi:(P,\mu_1,\eta_1,c_1)\longrightarrow (Q,\mu_2,\eta_2,c_2)$ be an H-homomor-phism, then $ker\varphi$ is a normal sub-H-group of $P$.
\end{proposition}
\begin{proof}
Let $g_1,g_2\in ker\vf$, and let $\al_1,\al_2$ be paths from $\vf(g_1)$ and $\vf(g_2)$ with end in $Q_0$, respectively. Then $\mu_2(\al_1,\al_2)$ is a path from $\mu_2(\vf(g_1),\vf(g_2))$ with end in $Q_0$. By $\mu_2\circ(\vf\times\vf)\simeq \vf\circ\mu_1$, there exists a path from $\vf(g_1.g_2)$ with end in $Q_0$ which implies that $ker\vf$ is closed under multiplication of $P$. Similarly $ker\vf$ is closed under inversion of $P$. By definition $ker\vf$ is saturated that imply by Theorem 3.11, $ker\vf$ is a sub-H-group of $P$. Normality of $ker\vf$ follows from associativity and $\vf\circ\eta_1\simeq\eta_2\circ\vf$.
\end{proof}
Let $\varphi:P\longrightarrow Q$ be an H-homomorphism, $A\sub P$ and $B\sub P'$, then
$$\widetilde{\varphi(A)}=\{q\in Q\ |\ q\widetilde{\in}\varphi(A)\},$$
$$\widetilde{\varphi^{-1}(B)}=\{p\in P\ |\ \varphi(p)\widetilde{\in}B \}.$$
Now we can state the following useful lemma.
\begin{lemma}Let $\varphi:(P,\mu_1,\eta_1,c_1)\longrightarrow (Q,\mu_2,\eta_2,c_2)$ be an H-homomorphism. Then\\
(i) If $(P',\mu_1',\eta_1',c_1')$ is a sub-H-group of $P$, then
$\widetilde{\varphi(P')}$ is a saturated sub-H-group of $Q$;\\
(ii) If $(Q',\mu_2',\eta_2',c_2')$ is a sub-H-group of $Q$, then
$\widetilde{\varphi^{-1}(Q')}$ is a sub-H-group of P. If $Q'$ is
normal, then so is $\widetilde{\varphi^{-1}(Q')}$.
\end{lemma}
\begin{proof}
By a similar proof of Proposition 4.2 the results hold.
\end{proof}
Suppose that N is a normal sub-H-group of $P$, $\varphi$ is an
H-homomorphism from $P$ to $Q$ and $\pi$ is the natural map from
$P$ to $P/N$. We would like to find an H-homomorphism
$\overline{\varphi}:P/N\longrightarrow Q$ such that $\overline{\varphi}(gN)=\varphi(g)$. But there is no meaning of H-homomorphism for $\overline{\varphi}$ because $P/N$ is not necessarily an H-group. Although we can assume every abstract group as a topological group by discrete topology, but it is prevalent that topology of $P/N$ must be related to the topology of $P$. By using the functor $\pi_0$, we overcome this problem and have some results as follow in the category of groups. In section 5 we will endow $P/N$ by quotient topology induced from $P$ and prove that $P/N$ by this topology is quasitopological group in the sense of \cite{A} and the following results hold in the category of quasitopological group.\\
For canonical map $\pi:P\lo P/N$ let $\ov{\pi}:\pi_0(P)\lo P/N$ by $\overline{\pi}([g])=gN$. Here is the key result.
\begin{theorem}
For any H-homomorphism $\varphi$ whose kernel K contains a normal sub-H-group $N$ of $P$, $\pi_0(\vf)$ can be factored
through $P/N$. In other words, there is a unique
homomorphism $\overline{\pi_0(\varphi)}:P/N\longrightarrow
\pi_0(Q)$,
such that $\overline{\pi_0(\varphi)}\circ \overline{\pi}=\pi_0(\varphi)$, i.e., the following diagram is commutative:
$$\xymatrix{
\pi_0(P) \ar[r]^{\pi_0(\vf)} \ar[d]_{\ov\pi}
& \pi_0(Q)  \\
\fr{P}{N}. \ar@{-->}[ur]_{\ov{\pi_0(\vf)}}  }$$
Furthermore,\\
(i) $\overline{\pi_0(\varphi)}$ is an epimorphism if $\pi_0(\varphi)$ is onto;\\
(ii) $\overline{\pi_0(\varphi)}$ is a monomorphism if and only if $K = \wt{N}$.
\end{theorem}
\begin{proof}
(i) It follows from commutativity of diagram. \\
(ii) Assume $\overline{\pi_0(\varphi)}$ is monomorphism. Since $K$ is saturated and contains $N$, we have $\wt{N}\sub K$. Let $g\in K$, then $\vf(g)\hin Q_0$ and so $\overline{\pi_0(\varphi)}(gN)=1$. By injectivity of $\overline{\pi_0(\varphi)}$, $gN=\wt{N}$ and therefore $K=\wt{N}$. The converse is trivial.\\
\end{proof}
The factor theorem yields a fundamental result .
\begin{theorem}
(The First H-isomorphism Theorem). If $\varphi:P\longrightarrow Q$ is an H-homomorphism with kernel $K$,
then $\pi_0(\wt{\varphi(P)})$ is isomorphic to $P/K$.
\end{theorem}
\begin{proof}
Consider $\theta:P/K\lo \pi_0(\wt{\vf(P)})$ by $\theta (gK)=[\vf(g)]$. Since $\vf$ is an H-homomorphism, $\theta$ is well defined and homomorphism. For $x=[q]\in\pi_0(\wt{\varphi(P)})$ there exist $p\in P$ such that $q\hin \{\vf(p)\}$ and $\theta(pK)=[\vf(p)]=[q]=x$. Hence $\theta$ is onto. Also, if $\theta(gK)=[\vf(g)]=1$, then $\vf(g)\in Q_0$ and $\theta$ is injective.
\end{proof}
 If $M$ and $N$ are saturated sub-H-groups of $P$, $G_1=\pi_0(M)$ and $G_2=\pi_0(N)$, then by using Theorem 3.13, $M\cap N$  is $P_{G_1\cap G_2}$ that is a sub-H-group of $P$.
\begin{lemma}
Let $M$ and $N$ be saturated sub-H-groups of $P$ and $N\unlhd P$. Then\\
(i) $\wt{MN} =\wt{ NM}$, and $\wt{MN}$ is a sub-H-group of $P$;\\
(ii) N is a normal sub-H-group of $\wt{MN}$;\\
(iii) $M\cap N$ is a normal sub-H-group of $M$.
\end{lemma}
\begin{proof}
Lemma 3.16, Proposition 3.14 and normality of $N$ imply (i). Since $N$ and $\wt{MN}$ are saturated and $\pi_0(\wt{MN})=\pi_0(M).\pi_0(N)$, $\pi_0(N)$ is a normal sub group of $\pi_0(M).\pi_0(N)$, by Corollary 3.13 and (i), which implies $N$ is a normal sub-H-group of $\wt{MN}$. The proof of (iii) is similar to (ii).
\end{proof}
\begin{theorem}
(The Second H-isomorphism Theorem). If $M$ and $N$ are saturated sub-H-groups of $P$ and $N\unlhd P$, then $$M/{M\cap N}\cong {\wt{MN}}/{N}.$$
\end{theorem}
\begin{proof}
Define $\theta:{M}/{M\cap N}\lo {\wt{MN}}/{N}$ by $\theta(g(M\cap N))=gN$. Obviously $\theta$ is a well defined homomorphism.
If $\theta(g(M\cap N))=gN=1$, then $g\hin N$ (equivalently $g\in N$ since $N$ is saturated) and hence $\theta$ is a monomorphism.
Assume $gN\in \wt{MN}$. By definition of $\wt{MN}$, there exist $mn\in MN$ such that $g\hin \{mn\}$. Hence $\theta(m(M\cap N))=mN=mnN=gN$ which implies $\theta$ is an epimorphism.
\end{proof}
\begin{theorem}
(The Third H-isomorphism Theorem). If $N$ and $M$ are normal saturated sub-H-groups of $P$ and $N$ is contained in $M$, then $${P}/{M}\cong \frac{{P}/{N}}{{M}/{N}}.$$
\end{theorem}
\begin{proof}
Define $\theta:{P}/{N}\lo{P}/{M}$ by $\theta(gN)=gM$ which is an epimorphism with kernel ${M}/{N}$.
\end{proof}
Now suppose that $N$ is a normal sub-H-group of $P$. If $M$ is a saturated sub-H-group of $P$ containing
N, there is a natural analogue of $M$ in the quotient H-group ${P}/{N}$, namely the subgroup ${M}/{N}$ .
In fact we can make this correspondence precisely. Let $\Psi$ be a map from the set of
saturated sub-H-groups of $P$ containing $N$ to the set of subgroups of ${P}/{N}$ by $\Psi(M)={M}/{N}$. We
claim that $\Psi$ is a bijection. For, if ${M_1}/{N}={M_2}/{N}$, then for any $m_1\in M_1$, we have $m_1N=m_2N$,
for some $m_2\in M_2$, so that $m_2^{-1}m_1\hin N$ which is contained in $M_2$. Thus $M_1\subseteq M_2$, and by
symmetry the reverse inclusion holds, so that $M_1=M_2$ and $\Psi$ is injective. Now if $G$ is a
subgroup of $\frac{P}{N}$ and $\pi:P\longrightarrow P/N$ is canonical, then
$${\pi^{-1}(G)}=\{p\in P\ |\ pN\in G\}$$ is a saturated sub-H-group of $P$ containing $N$, and $\Psi({\pi^{-1}(G)})=\{pN\ |\ pN\in G\}=G$
proving surjectivity of $\Psi$.
The map $\Psi$ has a number of other interesting properties, summarized in the following
result.
\begin{theorem}
(The Correspondence Theorem). If $N$ is a normal sub-H-group of $P$, then the above map $\Psi$ sets up a one-to-one correspondence between saturated sub-H-groups of $P$ containing $N$ and
subgroups of ${P}/{N}$. The inverse of $\Psi$ is the map $\Phi:G\mapsto {\pi^{-1}(G)}$, where $\pi$ is the canonical
map of $P$ to ${P}/{N}$. Furthermore,\\
(i) $M_1$ is a sub-H-group of $M_2$ if and only if ${M_1}/{N}\leq {M_2}/{N}$ , and in this case we have
$$[M_2 : M_1 ]=[{M_2}/{N} : {M_1}/{N}];$$
(ii) If $M$ is a normal sub-H-group of $P$, then ${M}/{N}$ is a normal subgroup of ${P}/{N}$;\\
(iii) $M_1$ is a normal sub-H-group of $M_2$ if and only if ${M_1}/{N}$ is a normal subgroup of ${M_2}/{N}$ , and
in this case,$${M_2}/{M_1}\cong \frac{{M_2}/{N}}{{M_1}/{N}}.$$
\end{theorem}
We introduced monics, epics and H-homomorphisms in
$hTop_*$ in section one. Now we define H-morphisms.
\begin{definition}
(i) An {\it H-monomorphism} is a monic H-homomorphism. \\
(ii) An {\it H-epimorphism} is an epic H-epimorphism.\\
(iii) An {\it H-endomorphism} is an H-homomorphism of an H-group to itself.\\
(iv) An {\it H-automorphism} is an H-isomorphism of an H-group to itself.
\end{definition}
Now we introduce a family of H-isomorphisms that makes a group.\\
\begin{proposition}
Let $\varphi_a:P\longrightarrow P$ given by $\varphi_a(g)=a.(g.a^{-1})$, then for each $a\in P$, $\vf_a$ is an H-isomorphism .
\end{proposition}
\begin{proof}
First we show that $\vf_a$ is an H-homomorphism. Let $c$ be homotopic to $\mu(1,\eta)$ by $H_1$ and $H_2$ is the homotopy
between $1_P$ and $\mu(1,p_0)$. \\
Let $\xi_1,\ \xi_2,\ \zeta_1,\ \zeta_2:P\times P\lo P$ by $$\xi_1(g,g')=a.(((g.p_o).g').a^{-1}),\ \ \ \ \ \ \ \  \xi_2(g,g')=(a.g).(p_0.(g'.a^{-1})),$$ $$\zeta_1(g,g')=(a.g).((a.a^{-1}).(g'.a^{-1})), \ \ \ \ \zeta_2(g,g')=(a.g.a^{-1}).(a.g'.a^{-1}).$$ By associativity of $\mu$, there exist homotopies $H_3,\ H_4$ such that  $\xi_1\simeq\xi_2$ by $H_3$ and $\zeta_1\simeq\zeta_2$ by $H_4$. Define $F:P\times P\times I\lo P$ by\\
\begin{displaymath}
{F}(g,g',t)= \left\{
\begin{array}{lr}
a.((H_2(g,2t).g').a^{-1})      &       0\leq t\leq 1/2,         \\
H_3(g,g',4t-2)                   &       1/2\leq t\leq 3/4,        \\
(a.g).(H_1(a,8t-6).(g'.a^{-1}))  &       3/4\leq t\leq 7/8,         \\
H_4(g,g',8t-7)                   &        7/8\leq t\leq 1.
\end{array}
\right.
\end{displaymath}
that $F(g,g',0)=a.((g.g').a^{-1})$ and $F(g,g',1)=(a.g.a^{-1}).(a.g'.a^{-1})$. Therefore $\vf_a\circ\mu\simeq \mu\circ(\vf_a\times\vf_a)$. Similarly
$\eta\circ\vf_a\simeq\vf_a\circ\eta$ which implies $\vf_a$ is an H-homomorphism. Associativity of $\mu$ implies that $\vf_a\circ\vf_{a^{-1}}\simeq 1$ and $\vf_{a^{-1}}\circ\vf_a\simeq 1$. Therefore $\vf_a$ is an H-isomorphism for every $a\in P$.
\end{proof}
\begin{xrem}
Note that $a.(g.a^{-1})$ and $(a.g).a^{-1}$ are different and $\vf_a(g)$ can not be shown by $a.g.a^{-1}$. Hence we define $\vf^a:P\longrightarrow P$ by $\vf^a(g)=(a.g).a^{-1}$ which is homotopic to $\vf_a$ by associativity.
\end{xrem}
\begin{definition}
We call the H-isomorphism $\vf_a$ introduced above an H-inner automorphism of $P$.
\end{definition}
As in Algebra we expect the equality $\vf_a\circ\vf_b=\vf_{a.b}$, but
$$\vf_a\circ\vf_b(g)=a.([b.(g.b^{-1})].a^{-1})$$ $$\vf_{a.b}(g)=(a.b).[g.(a.b)^{-1}].$$
This shows that $\vf_a\circ\vf_b\neq\vf_{a.b}$. If $\eta\mu\simeq\mu(\eta\times\eta)\circ T$, where $T:P\times P\lo P\times P$ by $T(x,y)=(y,x)$, then we have $\vf_a\circ\vf_b\simeq\vf_{a.b}$ by using associativity. It seems that for making a group of H-inner automorphism by composition as binary operation, we must work by classes of maps each of which are homotopic with $\vf_a$, $a\in P$ .
\begin{lemma}
If $a,b\in P$ are in the same path component, then $\vf_a\simeq\vf_b$ .
\end{lemma}
\begin{proof}
Let $\al$ be a path from $a$ to $b$. Then $F:P\times I\lo P$ by $F_t(g)=(\al(t)).(g.(\al(t))^{-1})$ is a homotopy between $\vf_a$ and $\vf_b$.
\end{proof}
\begin{definition}
We define the center of an H-group $P$ as follows: $$Z(P)=\{g\in P|\ \mu(g,-)\simeq\mu(-,g)\}.$$
\end{definition}
\begin{theorem}
$Z(P)$ is a normal saturated sub-H-group of $P$.
\end{theorem}
\begin{proof}
For $g_1,g_2\in Z(P)$, $g_1.g_2\in Z(P)$ by using associativity. Similarly $Z(P)$ is closed under inversion. Assume $g_1\in Z(P)$, then $L_{g_1}=\mu(g_1,-)\simeq\mu(-,g_1)=R_{g_1}$. If $g_2$ is connected to $g_1$ by a path $\al$, then $L_{g_2}\simeq L_{g_1}$ and $R_{g_2}\simeq R_{g_1}$. Thus $L_{g_2}\simeq R_{g_2}$. Therefore $Z(P)$ is saturated and by Theorem 3.11 is a sub-H-group of $P$.
\end{proof}

\begin{proposition}Let $P$ be an H-group such that satisfies $\eta\circ\mu\simeq\mu(\eta\times\eta)\circ T$, where $T:P\times P\lo P\times P$ by $T(x,y)=(y,x)$. If $a\st{Z(P)}{\sim}b$, then $\vf_a\simeq\vf_b$.
\end{proposition}
\begin{proof}
If $a\st{Z(P)}{\sim}b$, then there exists a path $\alpha$ from $a^{-1}.b$ to $z\in Z(P)$. Hence $a\alpha$ is a path from $b$ to $a.z$
 which implies $\vf_b\simeq\vf_{a.z}$. Also, $\vf_{a.z}(g)=a.z.(g.(a.z)^{-1})$ shows that $\vf_{a.z}\simeq\vf_a$, since $(a.z).((-).(a.z)^{-1})\simeq$\\ $(a.z).((-).(z^{-1}.a^{-1}))\simeq (a.(-)).(z.(z^{-1}.a^{-1}))\simeq a.((-).a^{-1})$.
\end{proof}
\begin{xrem}
It seems that all properties of topological groups hold up to homotopy for H-groups and therefore the hypothesis $\eta\circ\mu\simeq\mu(\eta\times\eta)\circ T$ holds for every H-group, but authors did not find any proof for this. Note that this equivalence holds for loop spaces and for all CW H-groups, since every CW H-group is as homotopy type of a loop space \cite{S}.
\end{xrem}
Let $Inn(P)=\{\vf_a\ |\ a\in P\} $ and define an equivalence relation on it as follow: $$\vf_a\sim\vf_b\Leftrightarrow \vf_a\simeq\vf_b.$$
Put $HInn(P)=Inn(P)/\sim$ which is precisely $\{[\vf_a]\ |\ a\in P\}$. Then we have
\begin{theorem}Let $P$ be an H-group such that satisfies $\eta\mu\simeq\mu(\eta\times\eta)\circ T$, then
$HInn(P)$ is a group.
\end{theorem}
\begin{proof}
Define $[\vf_a][\vf_b]=[\vf_{a.b}]$ and $[\vf_a]^{-1}=[\vf_{a^{-1}}]$. Multiplication is well defined, since if $[\vf_a]=[\vf_{a'}]$ and $[\vf_b]=[\vf_{b'}]$, then $\vf_a\simeq\vf_{a'}$ and $\vf_b\simeq\vf_{b'}$. Therefore $\vf_a\circ\vf_b\simeq\vf_{a'}\circ\vf_{b'}$, but as we have shown $\vf_a\circ\vf_b\simeq\vf_{a.b}$, hence $\vf_{a.b}\simeq\vf_{a'.b'}$ which means $[\vf_{a.b}]=[\vf_{a'.b'}]$. Other properties of multiplication comes from homotopy associativity, homotopy invertibility and homotopy identity properties of $\mu$.
\end{proof}
\begin{theorem}Let $P$ be an H-group such that satisfies $\eta\mu\simeq\mu(\eta\times\eta)\circ T$, then
$HInn(P)\cong {P}/{Z(P)}$
\end{theorem}
\begin{proof}
Define $\Theta:{P}/{Z(P)}\longrightarrow HInn(P)$ by $\Theta(aZ(P))=[\vf_a]$. By Proposition 4.17 $\Theta$ is well defined and obviously $\Theta$ is onto and homomorphism. For injectivity, assume that $\vf_a\simeq\vf_b$ or equivalently $\vf_{a.b^{-1}}\simeq 1_P$. If $F$ is the homotopy between $1_P$ and $\vf_{a.b^{-1}}$, then by using associativity of multiplication, $\mu(b.a^{-1},-)$ and $\mu(-,b.a^{-1})$ are homotopic by $(b.a^{-1})F$. Hence $b.a^{-1}\in Z(P)$ which implies $aZ(P)=bZ(P)$.
\end{proof}
\section{topological view}
In this section $(P,\mu,\eta,c)$ is an H-group , $(P',\mu',\eta',c')$ is a sub-H-group of $P$ and 
${P}/{P'}$ is the set of all left cosets of $P'$ in $P$.
 We intend to topologized the set ${P}/{P'}$ by quotient topology that is induced by canonical map $\pi:P\lo{P}/{P'}$ which makes it a quasitopological group.
Also we study path component space of H-groups and find out a necessary and sufficient condition for significance of 0-semilocally simply connectedness introduced in \cite{BR}.

As introduced in \cite{H}, the path component space of a topological space $X$ is $\pi_0(X)$ with the quotient topology with respect to the quotient map $q':X\lo\pi_0(X)$, where $q'(x)=[x]$ which is denoted by $\pi_0^{top}(X)$.
\begin{definition}
A space X is 0-semilocally simply connected (0-SLSC) if for each point $x \in X$, there is an open
neighborhood $U$ of $x$ such that the inclusion $i : U \lo X$ induces the constant map $\pi_0(i) : \pi_0^{top}
(U) \lo\pi_0^{top}(X)$.
\end{definition}
\begin{proposition}
A space X is 0-semilocally simply connected if and only if each path component of $X$ is open.
\end{proposition}
\begin{proof}
Let $X=\bigsqcup_{i\in I}X_i$, where $X_i$'s are path component of $X$. For an arbitrary $x$ there is $j\in I$ such that $x\in X_j$. Since
$X$ is 0-semilocally simply connected, there exists an open neighborhood $U$ of $x$ such that $\pi_0(i) : \pi_0^{top}(U) \lo\pi_0^{top}(X)$ is constant map, or equivalently $U$ meets just one path component of $X$ which implies $U\sub X_j$. Conversely, if each path component of $X$ is open, let $U$ be the path component containing $x$.
\end{proof}
\begin{xrem}
Obviously locall path connectivity follows 0-semilocally simply connectedness.
\end{xrem}

\begin{corollary}
X is 0-SLSC if and only if $\pi_0^{top}(X)$ has the discrete topology.
\end{corollary}
 By Lemma 3.22, $\theta:{P}/{P_0}\lo\pi_0(P)$ by $\theta(gP_0)=[g]$ is a group isomorphism. Then we have
\begin{lemma}
The group isomorphism $\theta:{P}/{P_0}\lo\pi_0^{top}(P)$ is a homeomorphism.
\end{lemma}
\begin{proof}
Consider the following commutative diagram:
$$\xymatrix{
P \ar[r]^{1} \ar[d]_{q}
& P\ar[d]^{q'}  \\
{P}/{P_0} \ar[r]^{\theta} & \pi_0^{top}(P).  }$$
Since $q$ and $q'$ are quotient maps, the result holds.
\end{proof}
Some important facts about the canonical map $\pi$ are collected in the following.
\begin{proposition}
Let $(P,\mu,\eta,c)$ be an H-group and $(P',\mu',\eta',c')$ be a sub-H-group of $P$. Let ${P}/{P'}$ be the set of all left cosets of $P'$ endowed
with the quotient topology induced from $P$ by $\pi:P\lo P/P'$. Then\\
(i) $\pi$ is onto; \\
(ii) $\pi$ is continuous;\\
(iii) If $P$ is 0-semilocally simply connected, then $\pi$ is open .
\end{proposition}
\begin{proof}
(i) is obvious and (ii) follows by the definition of quotient topology. For (iii), let $U$ be open in $P$. We must show that $\pi(U)$
is open in ${P}/{P'}$ i.e. $\pi^{-1}(\pi(U))$ is open in $P$. We have $\pi^{-1}(\pi(U))=\wt{U}={\bigcup \atop \al\in J}O_{\al}$, where $O_{\al}$'s are path components of $P$ that intersect $U$. 0-semilocally connectivity of $P$ implies that $O_{\al}$'s are open and hence $\pi^{-1}(\pi(U))$ is open, as desired.
\end{proof}
\begin{theorem}
Let $(P,\mu,\eta,c)$ be an H-group and $(N,\mu',\eta',c')$ be a normal sub-H-group of $P$.
Then ${P}/{N}$ endowed with the quotient topology induced from $P$ by $\pi:P\lo P/N$ is a homogeneous space.
\end{theorem}
\begin{proof}
Let $g_1N,g_2N\in {P}/{N}$ and $a\in P$ such that $(aN).(g_1N)=g_2N$ (namely $a=g_2g_1^{-1}$). Define the mapping $ L_{aN}:{P}/{N}\lo{P}/{N}$ by $L_{aN}(gN)=(aN)(gN)=(ag)N$. Then it is easy to check that $ L_{aN}$ is well-defined mapping of ${P}/{N}$ onto itself.\\
Continuity of $L_{aN}$ comes from continuity of $L_a:P\lo P$ and universal property of quotient map $\pi$.  Applying the previous argument to $a^{-1}$ we get ${L_{aN}}^{-1}=L_{a^{-1}N}$ which is continuous, hence $ L_{aN}$ is a homeomorphism. Therefore ${P}/{N}$ acts on itself by left and right translation ($R_{aN}(gN)=(ga)N$) as a group of self homeomorphisms. Clearly these actions are both transitive, and hence the result holds.
\end{proof}
\begin{xrem}
Note that $L_a$ is not necessarily a homeomorphism because $L_a\circ L_a^{-1}$ is homotopic to $1_P$ and is not equal to $1_P$. But fortunately $L_a$'s are homotopy equivalence.
\end{xrem}
\begin{definition}(\cite{A})
A quasitopological group G is a group with topology such that inversion and all translation are continuous.
\end{definition}
\begin{theorem}
Let $(P,\mu,\eta,C)$ be an H-group and $(N,\mu',\eta',C')$ be a normal sub-H-group of $P$.
Then ${P}/{N}$ endowed with the quotient topology induced from $P$ by $\pi:P\lo P/N$ is a quasitopological group.
\end{theorem}
\begin{proof}
It was proved in the previous theorem that all translations are continuous. Continuity of inversion follows from the universal property of quotient map $q:P\lo{P}/{N}$ and continuity of $\eta$.
\end{proof}
\begin{corollary}
$\pi_0^{top}$ is a functor from the category of H-groups to the category of quasitopological group.
\end{corollary}
\begin{theorem}
Let $P$ be a 0-SLSC H-group with $N$ as normal sub-H-group. Then \\
(i) The canonical mapping $\pi:P\lo{P}/{N}$ is a continuous and open homomorphism; \\
(ii) ${P}/{N}$ endowed with quotient topology is a topological group.
\end{theorem}
\begin{proof}
(i) By Proposition 5.5, $\pi$ is clearly a continuous and open map. It is enough to show that $\pi$ is a homomorphism. For this let $g,g'\in P$. Then
$$\pi(gg')=gg'N=(gN)(g'N)=\pi(g)\pi(g'),$$ since $N$ is a normal sub-H-group of $P$.\\
(ii) By (i) $\pi$ is open map which implies $\pi\times\pi$ is quotient  map and hence multiplication in ${P}/{N}$ is continuous. Using  Theorem 5.10, the result holds.
\end{proof}
\begin{theorem}
Let $P$ be a 0-SLSC H-group with normal sub-H-group $N$, then ${P}/{N}$ is a discrete topological group.
\end{theorem}
\begin{proof}
Since $P$ is 0-SLSC, $\wt{N}$ is open which implies that the identity element in the topological group $P/N$ is open, hence $P/N$ has discrete topology.
\end{proof}
\begin{xrem}
If we consider quotients of H-groups and path component spaces by quotient topology as described above, then all of group homomorphism and group isomorphism in section 4 hold in the category of quasitopological group and continuous homomorphism.
\end{xrem}
\section{applications to topological homotopy groups}
In 2002, a work of D. K. Biss \cite{B} initiated the development
of a theory in which the familiar fundamental group $\pi_1(X,x)$ of
a topological space $X$ becomes a topological space denoted by
$\pi_1^{top}(X,x)$ by endowing it with the quotient topology inherited
from the path components of based loops in $X$ with the compact-open
topology. Among other things, Biss claimed that $\pi_1^{top}(X,x)$
is a topological group. However, there
is a gap in the proof of Proposition 3.1 in \cite{B}. For more details, see
\cite{CM}. P. Fabel \cite{F} and Brazas \cite{BR}, discovered some interesting counterexamples for continuity of multiplication in $\pi_1^{top}(X,x)$.

In this section we provide some new results in topological homotopy groups and topological fundamental groups.
In \cite{G1}, the third author et.al. extended the above theory to higher homotopy
groups by introducing a topology on the n-th homotopy group of a
pointed space $(X,x)$ as a quotient of the n-loop space $\Omega^n(X,x)$
equipped with the compact-open topology. Call this space the topological
homotopy group and denote it by $\pi_n^{top}(X,x)$. The misstep
in the proof is repeated to prove that $\pi_n^{top}(X,x)$ is a topological group \cite[Theorem 2.1]{G1}. The third author et.al. in \cite{G2} showed that for n-Hawaiian like spaces, $n\geq 2$, $\pi_n^{top}(X,x)$ is a prodiscrete topological group.  Hence, there is an open question whether or not $\pi_n^{top}(X,x)$, $n\geq 1$, is a topological group.

Theorems 6.1, 6.2 and 6.4 in the following are reproved by using advantages of section 5 which can be found in \cite{B} and \cite{CM}.
\begin{theorem}
If $X$ is a path connected topological space, then $\pi_n^{top}(X,x)\cong\pi_n^{top}(X,y)$ as quasitopological groups, for each $x,y\in X$ and $n\geq1$.
\end{theorem}
\begin{proof}
By Example 2.2, $\al^+$ is an H-isomorphism between $\Omega(X,x)$ and $\Omega(X,y)$, where $\al$ is a path from $x$ to $y$. Since $\pi_0^{top}$ is a functor, $\pi_0^{top}(\al^+)$ is an equivalence morphism in category of quasitopological groups. Therefore $\pi_1^{top}(X,x)\cong\pi_1^{top}(X,y)$.
Also $\al^+$ is a homotopy equivalence, hence $\Omega\al^+:\Omega(\Omega(X,x),e_x)\lo\Omega(\Omega(X,y),e_y)$ is an H-isomorphism and therefore a homotopy equivalence, where $e_z$ is constant loop at $z\in X$. Consider $\Omega^n$ as the composition of $\Omega$ with itself $n$ times for $n\in N$. For $n>1$, we can construct by induction H-isomorphisms $\Omega^n(\al^+):\Omega(\Omega^n(X,x),e_x)\lo\Omega(\Omega^n(X,y),e_y)$, where $\Omega^n(\al^+)(\lambda)=\Omega^{n-1}(\al^+)\circ\lambda$. Since $\pi_0^{top}$ is a functor, $\pi_0^{top}(\Omega^n(X,e_x))\cong\pi_0^{top}(\Omega^n(X,e_y))$. Therefore $\pi_n^{top}(X,x)=\pi_0^{top}(\Omega^n(X,x))\cong\pi_0^{top}(\Omega^n(X,y))=\pi_n^{top}(X,y)$, as desired.
\end{proof}
\begin{theorem}
For homotopically equivalent topological spaces $(X,x)$ and $(Y,y)$, $\pi_n^{top}(X,x)\cong\pi_n^{top}(Y,y)$ as quasitopological groups, for $n\geq 1$.
\end{theorem}
\begin{proof}
We know that for each $n\in N$, $\Omega_n$ is a functor from pointed topological spaces to category of H-groups and hence $\Omega_n$ sends equivalent objects to equivalent objects. Since $\pi_0^{top}$ is also a functor with value in category of quasitopological groups, $\pi_0^{top}(\Omega_n(X,x))\cong\pi_0^{top}(\Omega_n(Y,y))$, as desired.
\end{proof}
\begin{lemma}
For every locally path connected, semilocally simply connected space $X$, $\Omega (X,x)$ is locally path connected, for each $x\in X$.
\end{lemma}
\begin{proof}
Use the proof of Lemma 3.2 in \cite{CM}.
\end{proof}
\begin{theorem}
For every locally path connected space $X$, $\pi_1^{top}(X,x)$ is discrete topological group, for each $x\in X$ if and only if $X$ is semilocally simply connected.
\end{theorem}
\begin{proof}
Assume $X$ is semi locally simply connected, by Lemma 6.3, $\Omega(X,x)$ is locally path connected. Hence Corollary 5.3 implies that
 $ \pi_0^{top}(\Omega(X,x))\cong\pi_1^{top}(X,x)$ is a discrete topological group. For converse see \cite[Theorem 1]{CM}.
\end{proof}
H. Wada in \cite{W} showed that for every m-dimensional finite polyhedron $Y$ and locally n-connected space $X$, the mapping space $X^Y$ is locally (n-m)-connected. Therefore we have
\begin{theorem}
For every locally n-connected pointed space $(X,x)$, the loop space $\Omega (X,x)$ is locally (n-1)-connected.
\end{theorem}
In \cite{G1} it is shown that topological n-th homotopy group of every locally n-connected metric space is discrete topological group. In the following theorem we prove this result in general case, in fact without metricness.
\begin{theorem}
For every locally n-connected space $X$, $\pi_n^{top}(X,x)$ is a discrete topological group, for each $x\in X$.
\end{theorem}
\begin{proof}
By Theorem 6.5, $\Omega(X,x)$ is locally (n-1)-connected space and so $\Omega^n(X,x)$ is locally 0-connected or equivalently a locally path connected H-group. Also $\pi_0^{top}(\Omega^n(X,x))\cong\pi_n^{top}(X,x)$, thus $\pi_n^{top}(X,x)$ is a discrete topological group by Corollary 5.4.
\end{proof}
A topological space $X$ is called n-semilocally simply connected if for each $x\in X$ there exists an open neighborhood $U$ of $x$ for which any n-loop in $U$ is nullhomotopic in $X$. In \cite{G1} it is proved that for locally (n-1)-connected metric spaces, discreteness of $\pi_n^{top}(X,x)$ and n-semilocally connectivity of $X$ are equivalent. By using this and above theorem we have the same result without metricness.
\begin{corollary}
Suppose that $X$ is a locally (n-1)-connected space and $x\in X$. Then the following are equivalent:\\
i) $\pi_n^{top}(X,x)$ is discrete.\\
ii) $X$ is n-semilocally simply connected at $x$.
\end{corollary}
\begin{definition} (\cite{V})
A loop $\al:(I,\partial I)\lo (X,x)$ is called {\it small} if there exists a representative of the homotopy class $[\al]\in\pi_1(X,x)$
in every open neighborhood $U$ of $x$. A non-simply connected space X is called {\it small loop space} if for every $x\in X$, every loop $\al:(I,\partial I)\lo (X,x)$ is
small.
\end{definition}
Biss in \cite{B} showed that the topological fundamental group of Harmonic Archipelago has indiscrete topology. Z. Virk in \cite{V} introduced a class of spaces, named small loop spaces, and constructed an example of small loop spaces by using Harmonic Archipelago. In the next theorem we will show that topological fundamental group of every small loop space has indiscrete topology and so is a topological group.
A basic account of small loop spaces may be found in \cite{V}.
\begin{theorem}
The topological fundamental group of every small loop space has trivial topology.
\end{theorem}
\begin{proof}
Let $X$ be a small loop space and $x\in X$. If there exists an open subset $U$ of $\pi_1^{top}(X,x)$ such that $\varnothing\neq U\neq\pi_1^{top}(X,x)$, we can assume that $U$ contains $[e_x]$, the identity element of $\pi_1^{top}(X,x)$. Let $[\al]\in\pi_1^{top}(X,x)$ such that $[\al]\notin U$, then $q^{-1}(U)$ is an open neighborhood of $e_x$ in $\Omega(X,x)$ that does not contain $\al$.
 There is a basic open neighborhood of $e_x$ like $\bigcap_{i=1}^n<K_i,U_i>$ such that $e_x\in\bigcap_{i=1}^n<K_i,U_i>\sub q^{-1}(U)$. Let $V=\bigcap_{i=1}^n U_i$, then $<I,V>\sub q^{-1}(U)$. Note that $V$ is nonempty open subset of $X$, since $x\in U_i$, for each i=1,2,...,n. By small loop property of $X$, there exists a loop $\al_V:I\lo V$ such that $[\al]=[\al_V]$. But $\al_V\in <I,V>$ implies that $[\al_V]=q(\al_V)\in U$, hence $[\al]=[\al_V]\in U$, which is a contradiction.
\end{proof}
An n-Hawaiian like space $X$, means the
natural inverse limit, $\underleftarrow{lim}(Y_i^{(n)},y_i^*)$, where $$(Y_i^{(n)},y_i^*)=\bigvee_{j\leq i}(X_j^{n},x_j^*)$$
is the wedge of $X_j^{(n)}$'s
in which $X_j^{(n)}$'s are (n-1)-connected,
locally (n-1)-connected, n-semilocally simply connected, and
compact CW spaces. The third author et.al. in \cite{G2} proved that the topological n-homotopy group of an n-Hawaiian like space is prodiscrete metrizable topological group for all $n>1$. Also, they proved in \cite{G1} that for a metric space $X$, $\pi_n^{top}(X,x)\cong \pi_1^{top}(\Omega^{n-1}(X,x),e_x)$. Since weak join of metric spaces is metric, n-Hawaiian like spaces are metric which implies that $\pi_1^{top}(Y,y)\cong\pi_n^{top}(X,x)$, where $Y$ is $\Omega^{n-1}(X,x)$ and $y$ is $e_x$ for n-Hawaiian like space $X$. Therefore we have a family of spaces with topological fundamental groups as topological groups.
\begin{theorem}
If $Y=\Omega^{n-1}(X,x)$, for n-Hawaiian like space $X$ and $n>1$, then $\pi_1^{top}(Y,y)$
is a topological group. Moreover, it is a prodiscrete metric space.
\end{theorem}
\subsection*{Acknowledgements}
This research was supported by a grant from Ferdowsi University of Mashhad.

\end{document}